\documentclass{amsart}

\usepackage{amssymb,amsmath,color}
\usepackage{amsthm,setspace}
\usepackage{url}
\usepackage{tikz,subcaption}
\usepackage[enableskew]{youngtab}
\usepackage{ytableau,varwidth}

\definecolor{red}{rgb}{1,0,0}
\definecolor{blue}{rgb}{.2,.2,.8}

\def\S{\mathcal S}

\def\N{\mathbb N}

\def\res{\mathcal{R}}
\def\val{\textrm{val}}

\newtheorem{theorem}{Theorem}[section]
\newtheorem{corollary}[theorem]{Corollary}
\newtheorem{proposition}[theorem]{Proposition}
\newtheorem{conjecture}{Conjecture}

\newtheorem{lemma}[theorem]{Lemma}

\theoremstyle{definition}
\newtheorem{definition}{Definition}
\newtheorem{example}{Example}
\newtheorem{remark}{Remark}

\begin{document}
\allowdisplaybreaks

	\title{Parity of $3$-regular partition numbers and Diophantine equations}
\author{Cristina Ballantine}	
	%\ and Mircea Merca\footnote{mircea.merca@profinfo.edu.ro}
	% \\ 
	%\small Department of Mathematics,
	%\small University of Craiova\\
	%\small Craiova, 200585 Romania}
	%\date{}
	\address{Department of Mathematics and Computer Science\\ College of The Holy Cross\\
	Worcester, MA 01610, USA }
	\email{cballant@holycross.edu}
	\author{Mircea Merca}
	\address{Department of Mathematical Methods and Models\\
	Fundamental Sciences Applied in Engineering Research Center\\  University Politehnica of Bucharest\\
RO-060042 Bucharest, Romania}
	\email{mircea.merca@profinfo.edu.ro}
	\author{Cristian-Silviu Radu} 
	\thanks{Cristian-Silviu Radu was supported by grant SFB F50-06 of the FWF}
	\address{Research Institute for Symbolic Computation (RISC), Johannes Kepler University\\
	 4040 Linz, Austria}
	\email{sradu@risc.uni-linz.ac.at}

	\maketitle

%\ccom{} \mcom{}

\begin{abstract}
Let $b_3(n)$ be  the number of $3$-regular partitions of $n$.  Recently, W. J. Keith and F. Zanello discovered
infinite families of Ramanujan type congruences modulo $2$ for $b_3(2n)$ involving every prime $p$ with $p \equiv 13, 17, 19, 23 \pmod {24}$,  and  O. X. M. Yao provided new infinite families of Ramanujan type congruences modulo $2$ for $b_3(2n)$  involving every prime $p\geqslant 5$.
In this paper, we introduce new infinite Ramanujan type congruences modulo $2$ for $b_3(2n)$. They complement naturally the results of Keith-Zanello and Yao and involve primes in  $\mathcal P=\{p \text{ prime }   :  \exists \,  j\in \{1,4,8\},\,  x, y \in \mathbb Z,\,  \gcd(x,y)=1  \text { with } x^2+216y^2=jp\}$ whose Dirichlet density is $1/6$.  As a key ingredient in our proof we show  that  of the number of primitive solutions for $x^2+216y^2=pm$, $p \in \mathcal P$, $p\nmid m$ and $pm\equiv 1\pmod{24}$, is divisible by $8$.  Here, the  difficulty  arises from the fact that $216$ is not idoneal.  We also give a conjectural exact formula for the number of solutions for this Diophantine equation. In the second part of the article,  we  study reversals of Euler-type identities. These are motivated by recent work of the second author on a reversal of Schur's identity which involves $3$-regular partitions weighted by the parity of their length. 


\end{abstract}

\section{Introduction}

A partition of a positive integer $n$ is a weakly decreasing sequence of positive integers whose sum is $n$. The positive integers in the sequence are called parts. For more on the theory of partitions, we refer the reader to \cite{Andrews98}. 

For an integer $\ell>1$  a partition is called $\ell$-regular if none of its parts is divisible by $\ell$. The number of the $\ell$-regular partitions of $n$ is usually denoted by $b_\ell(n)$ and its  arithmetic properties were  investigated extensively. See, for example, \cite{Carlson, Cui, Dand, Furcy, Hirschhorn, Lovejoy, Penn, Penn8, Xia, Xia14, Webb}.   
In classical representation theory, when $\ell$ is prime, $\ell$-regular partitions of $n$ parameterize the irreducible $\ell$-modular representations of the symmetric group $\S_n$  \cite{James}. 

Recently, W.~J. Keith and F. Zanello \cite{Keith} studied the parity of the coefficients of
certain eta-quotients and investigated the parity of $b_\ell(n)$ when $\ell \leqslant 28$. %They
%either established new parity results of these types where none were known, extend
%previous ones, or conjectured that such results are impossible. 
In particular, they
proved the following parity result for $b_3(n)$ which  involves every prime $p$ satisfying $p \equiv 13, 17, 19, 23 \pmod {24}$.

\begin{theorem}[Keith-Zanello]\label{Th1}
	The sequence $b_3(2n)$ is lacunary modulo $2$. If $p \equiv 13, 17, 19, 23 \pmod {24}$ is prime, then
	$$b_3\big(2(p^2n+pk-24^{-1}) \big)\equiv 0 \pmod 2$$
	for $1 \leqslant k \leqslant p - 1$, where $24^{-1}$ is taken modulo $p^2$.
\end{theorem}

Motivated by the Keith-Zanello result, O. X. M. Yao \cite{Yao} proved new infinite families of congruences modulo $2$ for $b_3(2n)$.  These new congruences involve any prime $p\geqslant 5$. 

\begin{theorem}[Yao]\label{Th2}
	Let $p\geqslant 5$ be a prime. 
	\begin{enumerate}
		\item [(1)] If $b_3\left( \frac{p^2-1}{12} \right) \equiv 1 \pmod 2$, then for $n,k\geqslant 0$
		\item[] $$b_3\left(2p^{4k+4}\,n+2p^{4k+3}\,j+\frac{p^{4k+4}-1}{12} \right) \equiv 0 \pmod 2$$
		where $1\leqslant j \leqslant p-1$ and for $n,k\geqslant 0$
		\item[] $$b_3\left( \frac{p^{4k}-1}{12} \right) \equiv 1 \pmod 2.$$
		\item [(2)] If $b_3\left( \frac{p^2-1}{12} \right) \equiv 0 \pmod 2$, then for $n,k\geqslant 0$ with $p \nmid (24n+1)$
		\item[] $$b_3\left(2p^{6k+2}\,n+\frac{p^{6k+2}-1}{12} \right) \equiv 0 \pmod 2$$
		and for $n,k\geqslant 0$
		\item[] $$b_3\left( \frac{p^{6k}-1}{12} \right) \equiv 1 \pmod 2.$$
	\end{enumerate}
\end{theorem}

In this paper, motivated by Theorems \ref{Th1} and \ref{Th2}, we provide new Ramanujan type congruences modulo $2$ for $b_3(2n)$. First we introduce some notation. Given the Diophantine equation $x^2+ny^2=m$, by a primitive  solution we mean a solution $(x,y)\in \mathbb Z^2$ satisfying $\gcd(x,y)=1$.  

For the remainder of the article $\mathcal P$ denotes  the set of primes $p$ such that for some $j \in \{1,4,8\}$ the  equation $x^2+24\cdot 9y^2=jp$ has primitive solutions.

\begin{theorem} \label{Tmain}
	For every $p\in \mathcal P$ and $n \geqslant 0$, we have $$b_3\big(2\,(p^2\,n+p\,\alpha-24_p^{-1})\big) \equiv 0 \pmod 2,$$
		where $0\leqslant \alpha < p$, $\alpha \neq \lfloor p/24 \rfloor$, $24_p^{-1}$ is the inverse of $24$ modulo $p$ taken such that $1\leqslant -24_p^{-1}\leqslant p-1$.  \end{theorem}
		
\begin{remark} \label{R1} In the proof of Theorem \ref{Th1}, the authors make use of the fact that if $p$ is a prime congruent to $13,17, 19, 23$ modulo $24$, then the Diophantine equation $x^2+24y^2=pm$ has no solutions. In section \ref{2.2} we show that, if $p\in \mathcal P$, then \begin{align*}
j=1 & \implies p \equiv 1\pmod{24}\\
j=4 & \implies p \equiv 7\pmod{24}
\\ j=8 & \implies p \equiv 5,11 \pmod{24}.\end{align*} For $p\in \mathcal P$, the  equation $x^2+24y^2=pm$ may have solutions and we need to understand the behavior modulo $8$ of the number of solutions with $3\nmid y$.  Here, $pm=24((p^2\,n+p\,\alpha-24_p^{-1}))+1$. 
\end{remark}
\begin{remark}\label{R2}
It will become clear from the discussion in section \ref{2.2} that for  $p\in \mathcal P$ we have $b_3\left( \frac{p^2-1}{12} \right) \equiv 1 \pmod 2$. 
\end{remark}
Thus, Theorem \ref{Tmain} complements naturally   Theorems \ref{Th1} and \ref{Th2}. 
\smallskip

A few words  about the proof of Theorem \ref{Tmain} are in order. It follows from \cite{Keith} that $b_3(2n)\equiv a(n)\pmod 2$, where $a(n)$ is the number of representations of $n$ as the sum of a generalized pentagonal number and a square not divisible by $3$.  Thus, to prove Theorem \ref{Tmain} we have to investigate the number of solutions $(x, y)\in \mathbb N^2$ for \begin{equation}\label{eqn_conv}\ x^2+24y^2=pm\end{equation} with $y=0$ or $3\nmid y$, and $pm$ as in Remark \ref{R1}. 
This in turn, requires a study of  the number  of primitive solutions for $x^2+216y^2=pm$ which we denote by $N_2(pm)$.  

Regarding \eqref{eqn_conv}, we note that there are four positive, reduced, primitive quadratic forms of discriminant $-96$. We denote by $N(-96, w)$ the number of primitive representations of $w$ by all these quadratic forms, i.e., \begin{equation*}
\begin{split}
N(-96,w):=&|\{(x,y)\in\mathbb{Z}:x^2+24y^2=w,\, \gcd(x,y)=1\}|\\
+&|\{(x,y)\in\mathbb{Z}:3x^2+8y^2=w,\, \gcd(x,y)=1\}|\\
+&|\{(x,y)\in\mathbb{Z}:5x^2+2xy+5y^2=w,\, \gcd(x,y)=1\}|\\
+&|\{(x,y)\in\mathbb{Z}:4x^2+4xy+7y^2=w,\, \gcd(x,y)=1\}|.
\end{split}
\end{equation*}
Then the mass formula (see, for example \cite[Lemma 3.25]{Cox}) gives \begin{equation*} N(-96,w)=2\prod_{p\mid w}\left(1+\Big(\frac{-6}{p}\Big)\right),\end{equation*} where the product is over all prime divisors of $w$ and $\big(\frac{-6}{p}\big)$ is the Jacobi symbol.

However,  since $\alpha^2\equiv 1\pmod{24}$ when $\gcd(\alpha,6)=1$, we see that, when $w\equiv 1\pmod{24}$, there are no integers $x,y$ such that $3x^2+8y^2=w$, $5x^2+2xy+5y^2=w$ or $4x^2+4xy+7y^2=w$. 
Therefore, if $w\equiv 1\pmod{24}$,
$$N(-96,w)=|\{(x,y)\in\mathbb{Z}:x^2+24y^2=w,\gcd(x,y)=1\}|.$$
In fact, as discussed in \cite[pg. 84]{D}, a positive  integer $w\equiv 1, 5, 7, 11 \pmod{24}$  can be represented by one and only one of the   four  quadratic forms above. In the language of genera, this means that the four quadratic forms lie in four different different genera.

In general, the number of positive, reduced, primitive quadratic forms  of a certain disciminant $-d$ is equal to the class number $h(-d)$. The quadratic form $x^2+216y^2$ has discriminant $-864$ and there are twelve positive, reduced, primitive quadratic forms with this discriminant.
The  total number of solutions $N(-864,w)$ for all these quadratic forms can again be  calculated  using \cite[Lemma 3.25]{Cox}. However, these quadratic forms do not lie in twelve different genera, i.e.,  we can not separate them  according to the congruence classes for $w$, regardless of the choice of modulus.  A number $n$ such that the different positive, reduced, primitive quadratic forms with  discriminant $-4n$ lie in different genera, is an idoneal number or Euler convenient number.  Thus, $216$ is not idoneal.

  Even though we can not obtain an explicit formula form $N_2(pm)$,  in section \ref{2.2} we are able to prove that, for $p\in\mathcal P$, $m\geq 1$ with $pm\equiv 1\pmod{24}$ and $p\nmid m$, 
$$N_2(pm)= 2\cdot |\{(a,b)\in\mathbb{Z}:a^2+216b^2=jm,\, \gcd(a,b)=1\}|,$$ where $j\in\{1, 4, 8\}$ is such that $x^2+216 y^2=jp$ has primitive solutions.
 This relation is enough to prove Theorem \ref{Tmain}. In section \ref{S4} we give a simple conjectured formula for $N_2(w)$.

The article is organized as follows. In section \ref{S2}, we prove Theorem \ref{Tmain}. In subsection \ref{2.1}, we prove the theorem for  small values of $p\in \mathcal P$ using  modular forms and   \cite[Lemma 4.5]{Radu1a}. We introduce all necessary notation and discuss several hypothesis of the lemma for general primes $p\geqslant 5$. Then, we show that the remaining hypotheses  hold for  $p\in \mathcal P$.  Lemma 4.5 in \cite{Radu1a}  produces a bound and reduces the proofs  to checking that the congruences hold for integers up to the given  bound. We   performed the required check.  However, as $p$ grows, it becomes impossible  to check the required finite but large number of congruences. In subsection \ref{2.2}, we give a proof of Theorem \ref{Tmain} (for all $p\in \mathcal P$) that uses only facts about quadratic forms and determine the Dirichlet density of the set $\mathcal P$. 
The results of section \ref{SL} are motivated by recent work of the second author \cite{M22}.  Recall that Schur's identity states that the number of partitions of $n$ into parts congruent to $\pm 1$ modulo $6$ equals the number of $3$-regular distinct partitions of $n$.   
We generalize the Schur reversal result  of \cite{M22} to all Euler-type identities, i.e., the number of partitions of $n$ with parts parts from a set $S_2$ is equal to the number of distinct partition of $n$ with parts from a set $S_1$. We prove analytically and combinatorially that, if $S_1, S_2$ are subsets of integers for which an Euler-type identity holds, then the number of distinct partitions with parts from $S_2$ weighted by the parity of the number of parts equals the number of  partitions with parts from $S_1$ weighted by the parity of the number of parts. In section \ref{S4}, we conjecture and exact formula for the number of primitive solutions for $x^2+216y^2=m$ when $m\equiv 1 \pmod{24}$. We end the section with companion conjectures for the parity of the number of $3$-regular partitions of $n$ with an even (respectively odd) number of parts.

\section{Proof of Theorem \ref{Tmain} } \label{S2}

We discovered Theorem \ref{Tmain} experimentally. Our initial guess was that the theorem holds for   \begin{align*}p\in \mathcal U:=\{29,59,79,103,223,& 227,241,251, 269, 293,  337, 419, 443, 487\}.\end{align*} For these values of $p$, the theorem can be proved using modular forms and \cite[Lemma 4.5]{Radu1a}. We present the proof here and follow with the proof for all $p \in \mathcal P$. The latter does not use modular forms.

First we introduce some necessary background on partitions. 

We denote by $p(n)$ the number of the partitions of $n$. For example, $p(5)=7$ since the  partitions of $5$ are 
\begin{equation*}
(5),\ (4,1),\ (3,2),\ (3,1,1),\ (2,2,1),\ (2,1,1,1),\ (1,1,1,1,1). %\label{p5}
\end{equation*}
 The generating function of $p(n)$ is given by
$$
\sum_{n=0}^\infty p(n)\, q^n = \frac{1}{(q;q)_\infty}.
$$
Here and throughout $q$ is a complex number with $|q| < 1$, and the symbol 
$(a;q)_\infty$ denotes the infinite product
$$(a;q)_\infty:= \prod_{n=0}^\infty (1-a q^n).$$ The generating function for the sequence $b_{\ell}(n)$ is given by
$$
\sum_{n=0}^\infty b_\ell(n) q^n= \frac{(q^\ell;q^\ell)_\infty}{(q;q)_\infty}.
$$

\subsection{Proof of Theorem \ref{Tmain} for $p\in \mathcal U$} \label{2.1}
Many proofs of congruences in the literature use \cite[Lemma 4.5]{Radu1a} and they introduce all necessary notation. The reader can consult, for example, \cite{RS1} and \cite{Chern}. For completeness and the convenience of the reader, we also define all needed notation. Let $\Gamma:=SL(2, \mathbb Z),$ and define  $$\Gamma_\infty:=\left\{\left(\begin{array}{cc}1& b \\ 0& 1\end{array}\right)\in \Gamma  \right\}.$$ For a positive integer $N$, we define the congruence subgroup $$\Gamma_0(N):=\left\{\left(\begin{array}{cc}a & b \\ c& d\end{array}\right)\in \Gamma : c\equiv 0 \pmod N \right\}.$$  If $M$ is a  positive integer, we write $R(M)$ for the set of finite integer sequences $r=(r_{\delta_1}, r_{\delta_2}, \ldots, r_{\delta_k})$, where $1={\delta_1}<{\delta_2}<\cdots <{\delta_k}=M$ are the positive divisors of $M$. For the remainder of the argument, we only consider positive divisors of a given integer.
Given a positive integer $m$, we denote by $S_{24m}$ the set of invertible quadratic residues modulo $24m$ and, for fixed $0\leqslant t\leqslant m-1$, we define
 $$P_{m,r}(t):=\left\{ts+\frac{s-1}{24}\sum_{\delta \mid M}\delta r_{\delta}\mod m :  s\in S_{24m}\right\},$$
 where by $x \mod m$ we mean the residue of $x$ modulo $m$. 
 
 Let $m,M$ and $N$ be positive integers.  Moreover, let $t$ be an integer such that $0\leqslant t\leqslant m-1$ and let $r=(r_\delta)\in R(M)$. We set $\kappa:=\gcd(1-m^2, 24)$ and write $\prod_{\delta\mid M}\delta^{|r_\delta|}=:2^s\ell$, where $s$ is a non-negative integer and $\ell$ is odd.  Then, we say that the tuple $(n,M,N,(r_\delta), t)\in \Delta^*$ if and only if all of the following six conditions are satisfied. 
 \begin{enumerate}
\item  $p\mid m$, $p$ prime, implies $p\mid N$; 
 
 \item  $\delta \mid M$, $\delta \geqslant 1$ such that $r_\delta\neq 0$ implies $\delta \mid mN$; 
 
 \item  $\kappa N \sum_{\delta \mid M}r_\delta \frac{mN}{\delta} \equiv 0 \pmod{24}$;
 
\item   $\kappa N \sum_{\delta \mid M}r_\delta \equiv 0 \pmod{8}$; 
 
\item  $\displaystyle  \frac{24m}{\gcd(\kappa(-24t-\sum_{\delta\mid M}\delta r_\delta),24m)}\Big| N$;

\item  If $2 \mid m$, then $(4\mid \kappa N \text{ and } 8\mid Ns)$ or $(2 \mid s \text{ and } 8 \mid N(1-j))$. 

\end{enumerate} 
  Finally, for $\gamma=\left(\begin{array}{cc}a & b \\ c& d\end{array}\right)\in \Gamma$, 
 we define $$p_{m,r}(\gamma):=\min_{d\in\{0, \ldots, m-1\}}\frac{1}{24}\sum_{\delta\mid M}r_\delta\frac{\gcd^2(\delta(a+\kappa dc),mc)}{\delta m}$$ and $$p^*_{r}(\gamma):=\frac{1}{24}\sum_{\delta\mid M}r_\delta\frac{\gcd^2(\delta,c)}{\delta}.$$
With all the notation in place, we can now state \cite[Lemma 4.5]{Radu1a} that will allow us to prove the congruences of Theorem \ref{Tmain} for  $p\in \mathcal U$. 

\begin{lemma} \label{radu} Let $u$ be a positive integer, $(m,M,N,t,r=(r_\delta))\in \Delta^*$, $a=(a_\delta)\in R(N)$.    Let $\{\gamma_1, \ldots, \gamma_n\}\subset \Gamma$ be a complete set of representatives of the double cosets in $\Gamma_0(N)\backslash \Gamma/\Gamma_\infty$. Assume that $p_{m,r}(\gamma_i)+p^*_s(\gamma_i)\geqslant 0$ for all $0\leqslant i\leqslant n$. Let $t_{\min}:=\min_{t'\in P_{m.r}(t)}t'$ and $$\nu:=\frac{1}{24}\left(\left(\sum_{\delta \mid N}a_\delta+ \sum_{\delta \mid M}r_\delta\right) [\Gamma:\Gamma_0(N)]-\sum_{\delta \mid N}\delta a_\delta\right)-\frac{1}{24m}\sum_{\delta \mid M}\delta r_\delta-\frac{t_{\min}}{m}.$$ Suppose $$\prod_{\delta\mid M}\prod_{n=1}^\infty (1-q^{\delta n})^{r_\delta}=\sum_{n=0}^\infty c_r(n)q^n.$$ If $$\sum_{n=0}^{\lfloor \nu\rfloor} c_r(mn+t')q^n\equiv 0 \pmod u, \ \ \text{ for all } t'\in P_{m,r}(t),$$ then $$\sum_{n=0}^{\infty} c_r(mn+t')q^n\equiv 0 \pmod u, \ \ \text{ for all } t'\in P_{m,r}(t).$$
\end{lemma}

%We prove the congruences of Theorem \ref{T1.4} algorithmically using   \cite[Lemma 4.5]{Radu1a}. We begin with some preliminaries. 

Recall that 
$$\sum_{n=0}^{\infty}b_3(n)q^n=\prod_{n=1}^{\infty}\frac{(1-q^{3n})}{(1-q^n)}. $$
Let 
$$\sum_{n=0}^\infty b(n) q^n:=\prod_{n= 1}^\infty \frac{ (1-q^n)^4}{1-q^{3n}}$$ In \cite[pg. 287]{Keith}  it is shown  that
$$ \sum_{n=0}^\infty b_3(2n)q^n\equiv  \sum_{n=0}^\infty b(n)q^n\pmod{2}.$$
We will prove below that   $$b(p^2\,n+p\alpha-24_p^{-1})\equiv 0\pmod{2}$$ where 
$p\in \mathcal U$ and  $\alpha\neq \lfloor p/24\rfloor$  is a residue modulo $p$.

To use the Lemma \ref{radu}, we write  $$\prod_{n= 1}^\infty \frac{ (1-q^n)^4}{1-q^{3n}}=\prod_{\delta\mid M}\prod_{n=1}^\infty (1-q^{\delta n})^{r_\delta}.$$

Since we want to prove congruences of the form $b(p^2\,n+p\alpha-24_p^{-1})\equiv 0\pmod{2}$ for certain primes $p$ and  integers $0\leqslant\alpha\leqslant p-1$, $p\neq \lfloor p/24\rfloor$,  we have $u=2$ and  $m=p^2$. Then, $\kappa=24$. We let $M=3$,  and $(r_{1},r_{3})=(4,-1)$.  
Since $\prod_{\delta\mid M}\delta^{|r_\delta |}=3$, we have $s=0$ and $\ell=3$. We take $N=3p$ and verify that for any $t=p\alpha-24_p^{-1}$ as in the statement of Theorem \ref{Tmain}  conditions 1.-6. are satisfied.

Since $$[\Gamma:\Gamma_0(N)]=N\prod_{x\mid N}(1+x^{-1}),$$ where the product is taken after all prime divisors of $N$, we have $[\Gamma:\Gamma_0(N)]=4(p+1).$ 

Since $N$ is square free, by \cite[Lemma 4.3]{Wang}, a complete set of representatives for  $\Gamma_0(N)\backslash \Gamma/\Gamma_\infty$ is given by $$\mathcal A_N=\left\{\left(\begin{array}{cc}1 & 0 \\ \delta & 1\end{array}\right): \delta \mid N, \delta \geqslant 1\right\}.$$ Here, the positive divisors of $N=3p$ are $\{1, 3,  p,  3p\}$.

We complete the proof of the congruences of Theorem \ref{Tmain} for $p\in \mathcal U$ case by case. For the discussion below we used Mathematica to  calculate all needed data.  Let $$\res_p:=\{p\alpha-24_p^{-1} : 0\leqslant \alpha<p, \ \alpha\neq \lfloor p/24\rfloor\}.$$
In each case below, we verified that $p_{m,r}(\gamma)>0$ for all $\gamma \in \mathcal A_N$ and thus we can take $a_\delta=0$ for all $\delta \mid N$. 
We give more details in the case $p=29$, including the values for $p_{m,r}(\gamma)$,  and omit the details in the other cases. 

  \noindent \underline{Case $p=29$.} We have $24_{p}^{-1}\equiv -6 \pmod{29}$ and $\lfloor p/24\rfloor=1$. We find that $$P(6)=\{6, 64, 151, 180, 209, 238, 296, 412, 499,  615,  673, 702, 731, 760\}$$ and $$P(3\cdot 29+6)=\{93, 122, 267, 325,  354,  383, 441,470,  528,  557, 586,  644, 789,   818\}.$$
 Thus,  $$P(6)\cup P(3\cdot 29+6)=\res_{29}.$$ Next we calculate $p_{m,r}(\gamma)$ for each $\gamma \in \mathcal A_N$. We obtain $$\left\{\frac{11}{60552}, \frac{1}{20184},\frac{11}{60552}, \frac{1}{20184} \right\},$$ where the terms are listed in increasing order of $\delta \mid N$ defining $\gamma$. 

Thus, for  \begin{align*}a: =(a_{1},a_{3},a_{5},a_{15}) =(0,0,0,0)\end{align*} we have $p_{m,r}(\gamma)+p^*_s(\gamma)\geqslant 0$ for all $\gamma\in \mathcal A_N.$

From Lemma \ref{radu}, we obtain $\lfloor \nu \rfloor=14$   when $t=6$ and also when $t=3\cdot 29+6=93$. 

 It is easily verified that for each $t\in\res_{29}$, we have $$ b(29^2n+t) \equiv 0\pmod 2 \text{ for } 0\leqslant n\leqslant 14.$$ 
Thus, by Lemma \ref{radu}, for each $t\in\res_{29}$, we have $$ b(29^2n+t)\equiv 0\pmod 2 \text{ for all } n\geqslant 0.$$

 For the other primes in $\mathcal U$, we list the relevant data needed for Lemma \ref{radu} as well as the obtained bound in the table below. In the table, $A$ is the smallest positive integer not equal to $\lfloor p/24\rfloor$ such that $P(-24_p^{-1})\cup P(Ap-24_p^{-1})=\res_{p}.$ In each case, we used Mathematica to verify that $p_{m,r}(\gamma)>0$ for each $\gamma \in \mathcal A_N$, and we used $a: =(a_{1},a_{3},a_{5},a_{15}) =(0,0,0,0)$ to calculate $\lfloor \nu \rfloor$. 
 
 $$\begin{array}{c||c|c|c|c}\quad p \quad &\quad -24_p^{-1}\quad& \quad\lfloor p/24\rfloor\quad& \quad A\quad & \quad\lfloor \nu \rfloor\quad\\ \hline \hline  59 & 27 & 2& 1& 29\\ \hline 79 & 23& 3& 1& 39\\  \hline 103 & 30 & 4& 1& 51 \\ \hline 223 & 65& 9& 3& 111 \\ \hline 227 & 104 & 9& 1& 113 \\ \hline 241 & 10 & 10& 3& 120 \\ \hline 251 & 115& 10 & 1& 125 \\ \hline 269 & 56 & 11 & 1& 134 \\ \hline 293 & 61 & 12 & 2 & 146 \\ \hline 337 & 14 & 14 & 3& 168 \\ \hline 419 & 192 & 17 & 1& 209 \\ \hline 443 & 203 & 18 & 1 & 221 \\ \hline 487 & 142 & 20 &  1 &  243 
 \end{array}$$

\smallskip

 In each case, we verified the congruences up to the bound $\lfloor \nu \rfloor$ given by Lemma \ref{radu}. Hence, for $p\in \mathcal U$,  the congruences hold for all $n\geqslant 0$. We needed to compute  $b_3(m)$ modulo $2$ for all $m\leqslant 2(487^2\cdot 243+486)=115\,265\,106$. Since this was computationally expensive, we had to optimize the algorithm generating $b_3(m)$ modulo $2$ several times.

\subsection{Proof of Theorem \ref{Tmain} for all $p\in \mathcal P$} \label{2.2}
From \cite{Keith}, we have \begin{align*}\sum_{n=0}^\infty b_3(2n)q^n & \equiv \sum_{n\in \mathbb Z}q^{\frac{n}{2}(3n-1)^2}\left( 1+\sum_{n\in \mathbb Z}q^{(3n-1)^2}
\right)\\ & \equiv \sum_{n\in \mathbb Z}q^{\frac{n}{2}(3n-1)}+ \sum_{n,m\in \mathbb Z}q^{\frac{n}{2}(3n-1)+ (3m-1)^2}\\ & =: \sum_{r=0}^\infty a(r)q^r,\end{align*} where $a(r)$ is the number of representations of $r$ as the sum of a generalized pentagonal number and a the square of a an integer $y$ such that $y \not \equiv 0 \pmod 3$ if $y \neq 0$. Thus, $$a(r)=|\{(n,m)\in \mathbb Z^2;\  \frac{n}{2}(3n-1)+ (3m-1)^2=r\}|+\delta_{r = \frac{n}{2}(3n-1)}.$$ Here, $\delta_\rho$ is the Kronecker delta function which equals $1$ if statement $\rho$ is true and $0$ otherwise.  We can rewrite $a(r)$ as  $$a(r)=|\{(x,y)\in \mathbb N^2;\ x^2+24y^2=24r+1,  y=0 \text{ or } 3\nmid y.\}|$$ 

Our goal is to prove that for $p\in \mathcal P$, $n\geqslant 0$, $0\leqslant\alpha<p$, $\alpha \neq \lfloor p/24\rfloor$, we have $a(p^2\,n+p\,\alpha-24_p^{-1})\equiv 0 \pmod 2.$ 
\smallskip

\subsubsection{Preliminary lemmas}

In this subsection, we  prove several lemmas which will be key ingredients in the proof of the theorem. 
\begin{lemma} \label{L2}For any prime $p\geqslant 5$ and $0\leqslant\alpha\leqslant p-1$, $\alpha \neq \lfloor \frac{p}{24}\rfloor$,   we have $p^2 \nmid 24(p\alpha-24^{-1}_p)+1$. 
\end{lemma}

\begin{proof}
Recall that $-24^{-1}_p$ as an integer between $1$ and $p-1$. Then, $24\cdot (-24^{-1}_p)=tp-1$ with $1\leqslant t \leqslant 23$. Thus, $24(p\alpha-24^{-1}_p)+1= 24p \alpha+tp$.

Suppose $p^2 \mid 24(p\alpha-24^{-1}_p)+1$. Then $24(p\alpha-24^{-1}_p)+1=kp^2$ for some $k\in \mathbb Z$. Since $p^2\equiv 1 \pmod{24}$, we must have $k\equiv 1\pmod{24}$. Since $1\leqslant\alpha <p$ and $1\leqslant-24^{-1}_p\leqslant p-1$, we have $1\leqslant p\alpha-24^{-1}_p\leqslant p^2-1$ and  we must have $k=1$. 
Then,  $24(p\alpha-24^{-1}_p)+1=p^2$, and thus $24\alpha+t=p$. Finally, since $1\leqslant t \leqslant 23$, we have  $\alpha= \lfloor \frac{p}{24}\rfloor$ which is a contradiction. 
\end{proof}

\begin{corollary}\label {C_pm} For any prime $p\geqslant 5$ and $0\leqslant\alpha <p$, $\alpha \neq \lfloor \frac{p}{24}\rfloor$, and  $s=p^2n+p\alpha-24^{-1}_p$.  Then, $24s+1=pm$, for some positive integer $m$ such that $p\nmid m$.
\end{corollary}

\begin{proof}We have  $24s+1=24p^2n+24p\alpha-24^{-1}_p\cdot 24+1$ and the result follows immediately form  Lemma \ref{L2}.
\end{proof}

Next, for $m$ coprime to $6$ and $j\in \{1,4,8\}$, we consider the possibilities for $m \pmod{24}$ when \begin{equation} \label{eqn_j} x^2+216 y^2=jm\end{equation} has  primitive solutions. 

\begin{lemma} \label{L_j=1} Let $m$ be a positive integer coprime to $6$. If equation \ref{eqn_j} with $j=1$ has  primitive solutions, then $m\equiv 1 \pmod{24}$. 
\end{lemma}
\begin{proof}
Let $m>0$, $\gcd(m,6)=1$  and let $x,y\in \mathbb Z$ with $\gcd(x,y)=1$ be such that  $x^2+216y^2=m$, Then,   $2, 3 \nmid x$ and thus $x^2\equiv 1 \pmod{24}$. It follows that  $m\equiv 1\pmod{24}$. 
\end{proof}
\begin{lemma} \label{L_j=4} Let $m$ be a positive integer coprime to $6$. If equation \ref{eqn_j} with $j=4$ has  primitive solutions, then $m\equiv 7 \pmod{24}$. 
\end{lemma}
\begin{proof} Let $m>0$, $\gcd(m,6)=1$  and let  $x,y\in \mathbb Z$ with $\gcd(x,y)=1$ be such that  $x^2+216y^2=4m$. Then $x=2z$ for some $z\in \mathbb Z$ and  we have $z^2+54y^2=m$. Therefore, $2, 3 \nmid z$ and thus $z^2\equiv 1 \pmod{24}$. Since $\gcd(x, y)=1$, it follows that $2\nmid y$ and $y^2\equiv 1, 9 \pmod{24}$. Then $m\equiv 1+6y^2 \equiv 7 \pmod{24}$.
\end{proof}

\begin{lemma} \label{L_j=8} Let $m$ be a positive integer coprime to $6$. If equation \ref{eqn_j} with $j=8$ has  primitive solutions, then $m\equiv 5,11 \pmod{24}$. Moreover, if $(x,y)$ is a primitive solution for \eqref{eqn_j} with $j=8$, then 
\begin{itemize}\item[(i)] if $m\equiv 5 \pmod{24}$, then $\val_2(x)=2$;
\item[(ii)] f $m\equiv 11 \pmod{24}$, then $\val_2(x)\geqslant 3$.
\end{itemize}
\end{lemma}

\begin{proof} Let $m>0$, $\gcd(m,6)=1$   and suppose  $x, y\in \mathbb Z$ are such that  $\gcd(x,y)=1$ and $x^2+216y^2=8m$. Then, $x=4z$ for some $z\in \mathbb Z$ and $$2z^2+27y^2=m.$$ Since $\gcd(m,6)=1$, it follows that $2 \nmid y$ and $3\nmid z$. Then, one can check that   $y^2 \equiv 1,9\pmod{24}$ and  $z^2 \equiv 1,4, 16\pmod{24}$. A case by case verification shows that if $z^2\equiv 1\pmod{24}$, and $y^2\equiv 1,9\pmod{24}$, then $m\equiv 5\pmod{24}$, and if $z^2\equiv 4,16\pmod{24}$, and $y^2\equiv 1,9\pmod{24}$, then $m\equiv 11\pmod{24}$. Thus, if $m\equiv 5 \pmod{24}$, then $\val_2(z)=0$ and thus $\val_2(x)=2$, and if $m\equiv 11 \pmod{24}$, then $\val_2(z)\geqslant 1$ and thus $\val_2(x)\geqslant 3$.
\end{proof}
In particular, the previous three lemmas show that, if $p\in \mathcal P$, then \begin{align*}
j=1 & \implies p \equiv 1\pmod{24}\\
j=4 & \implies p \equiv 7\pmod{24}
\\ j=8 & \implies p \equiv 5,11 \pmod{24},\end{align*} which was noted in Remark \ref{R1}.

\begin{lemma}  \label{Lmain} Let $p\in \mathcal P$ and let $m$ be a positive integer  coprime to  $p$ such that $pm\equiv 1\pmod{24}$.  Let $j\in \{1, 4, 8\}$ such that $x^2+216 y^2=jp$ has primitive solutions. Define  \begin{align*} X_{p,m}& := \{(x,y)\in \mathbb Z^2 : x^2+216y^2=pm, \gcd(x,y)=1\}\\  A_{m}& := \{(a,b)\in \mathbb Z^2 : a^2+216b^2=jm, \gcd(a,b)=1\}.\end{align*} Then $|X_{p,m}|=2|A_m|$.
\end{lemma}
\begin{proof}
Let $p\in \mathcal P$ and $j\in \{1, 4, 8\}$ be such that $x^2+216 y^2=jp$ has primitive solutions and fix $x_1, y_1\in \mathbb Z$ with $\gcd(x_1, y_1)=1$   satisfying $x_1^2+216 y_1^2=jp$. Let $m>1$ be  such that $p\nmid m$ and $pm \equiv 1 \pmod{24}$.  We define a mapping $f:X_{p,m}\to A_m$ by $f(x,y)=(a,b)$ as follows. Start with $(x,y)\in X_{p,m}$.
\begin{itemize} \item[(i)] If $x_1x-216y_1y \equiv 0 \pmod p$, one can show that $x_1y+y_1x\equiv 0 \pmod p$ and we define \begin{align*}a& := \frac{x_1x-216y_1y }{p}, \qquad b := \frac{x_1y+y_1x}{p}. \end{align*} 

 \item[(ii)] If $x_1x+216y_1y \equiv 0 \pmod p$, one can show that $x_1y-y_1x\equiv 0 \pmod p$ and we define \begin{align*}a& := \frac{x_1x+216y_1y }{p},\qquad b := \frac{x_1y-y_1x}{p}. \end{align*} 
\end{itemize}
Then,  in both cases $$a^2+216 b^2=\frac{(x_1^2+216y_1^2)(x^2+216y^2)}{p^2}=jm.$$ To see  that $\gcd(a,b)=1$, express $x$ and $y$ in terms of $a$ and $b$. In case (i) we have \begin{equation}\label{xy(i)} x=\frac{x_1a+216y_1b}{j}, \qquad y=\frac{bx_1-ay_1}{j},\end{equation} 
and in case (ii) we have \begin{equation}\label{xy(ii)} x=\frac{x_1a-216y_1b}{j}, \qquad y=\frac{bx_1+ay_1}{j}.\end{equation} Since $pm$ is odd, $x$ must be odd. If $j=4, 8$, then $x_1$ is even and since $\gcd(x_1, y_1)=1$, it follows that $y_1$ is odd. Therefore, $a$ is even and $b$ is odd, and  from \eqref{xy(i)} and \eqref{xy(ii)}, we have that $\gcd(a,b)\mid \gcd(x,y)$. Hence $\gcd(a,b)=1$ and $(a,b)\in A_m$.

Next, we show that $f$ is surjective and two-to-one. Start with $(a,b)\in A_m$ and define $x,y$ by the formulas given in \eqref{xy(i)} and $\tilde x, \tilde y$ by the formulas given in \eqref{xy(ii)}. It is straight forward to verify that \begin{align*} x^2+216y^2& =pm\\ \tilde x^2+216\tilde y^2& =pm. \end{align*}

As before, if $j=4$, then $2\mid a, x_1$. If $j=8$, then $4\mid a, x_1$. Thus, in all cases $x,\tilde x\in \mathbb Z$.

Since $\gcd(a,b)=\gcd(x_1, y_1)=1$, it follows that $2 \nmid b, y_1$. Then, if $j=4$, $y, \tilde y\in \mathbb Z$.

If $j=8$, by Lemma \ref{L_j=8}, $p\equiv 5, 11 \pmod{24}$. If $p\equiv 5\pmod{24}$, $\val_2(a)=\val_2(x_1)=2$. As above, $2 \nmid b, y_1$ and thus $y, \tilde y\in \mathbb Z$. If $p\equiv 5\pmod{24}$, we have $\val_2(a)$, $\val_2(x_1)\geqslant 3$ and  thus $y, \tilde y\in \mathbb Z$. 

  A similar discussion as above for $a, b$, shows that $\gcd(x,y)=\gcd(\tilde x, \tilde y)=1$.  
Thus, $(x,y), (\tilde x, \tilde y)\in X_{p,m}$. 

Moreover, one can check that \begin{align} \label{cong1} x_1x-216y_1y\equiv 0 \pmod p\\ \label{cong2}x_1\tilde x+216y_1\tilde y\equiv 0 \pmod p.\end{align} Hence, $f^{-1}(a,b)=\{(x,y), (\tilde x, \tilde y)\}$. 

It remains to show that $(x,y)\neq (\tilde x, \tilde y)$. Suppose $(x,y)= (\tilde x, \tilde y)$.
Then each pair $(x,y),(\tilde x, \tilde y)$ satisfies both congruences \eqref{cong1} and \eqref{cong2}. As a consequence, $x_1x\equiv 0\pmod p$ which contradicts $\gcd(x,y)=\gcd(x_1, y_1)=1$. 

This shows that $f$ is surjective and two-to-one.
\end{proof}

\begin{corollary} \label{Cmain} With $p, m, j$ and $X_{p,m}$ as in Lemma \ref{Lmain}, $|X_{p,m}|\equiv 0 \pmod 8$ if  $m>1$. If $m=1$, then $|X_{p,m}|=4$.  
\end{corollary}
\begin{proof} If $m=1$, then $p\equiv 1\pmod{24}$ and therefore $j=1$. Then $(1, 0)$ and  $(-1,0)$ are the only elements of $A_m$ and $|X_{p,m}|=2|A_m|=4$. If $m>1$ and  $(a,b)\in A_m$, it follows that $a, b\neq 0$. Then, $(a,b)\in A_m$ if and only if $(-a, -b), (-a, b), (a, -b)\in A_m$. Hence,  $|A_m|\equiv 0\pmod 4$, and by Lemma \ref{L_j=8}, $|X_{p,m}|\equiv 0 \pmod 8$.
\end{proof}

\subsubsection{Proof of Theorem \ref{Tmain}}

We  introduce the following notation. 
\begin{align*}
 M_1(w)& :=|\{(x,y) \in \mathbb Z^2:x^2+24y^2=w\}|\\
M_2(w)& :=|\{(x,y) \in \mathbb Z^2:x^2+24\cdot 9y^2=w\}|\\ N_1(w)& :=|\{(x,y) \in  M_1(w): {\gcd(x,y)=1}\}|\\
N_2(w)& :=|\{(x,y) \in  M_1(w):  \gcd(x,y)=1\}|. \end{align*}
{Note that if $u$ is square free, then $M_i(u)=N_i(u)$. }

From \cite[pg. 84]{D}, if $w\equiv 1 \pmod{24}$, we have \begin{equation} \label{M1} M_1(w)=2\sum_{d\mid w}\Big(\frac{-6}{d}\Big),\end{equation} where the sum is over all positive divisors of $w$  and $\big(\frac{-6}{p}\big)$ is the Jacobi symbol.  As discussed in the introduction, since $24$ is an idolean number,  if $w\equiv 1 \pmod{24}$,   we have \begin{equation}\label{N} N_1(w)= N(-96,w)=2\prod_{p\mid w}\left(1+\Big(\frac{-6}{p}\Big)\right),\end{equation} where the product is over all prime divisors of $w$.

%Moreover, if $w\not \equiv 1 \pmod{12}$, then $N_1(w)=0$.
In particular,  if $p\in \mathcal P$ and $p \equiv 1 \pmod{24}$, then $M_1(p)=N_1(p)=4$. %if  and $M_1(p)=N_1(p)=0$ if $p\not \equiv 1 \pmod{24}$.

\smallskip

Recall that $s=p^2n+p\alpha-24^{-1}_p$ and $24s+1=pm$ with $p\nmid m$. 
If $(x,y)\in \mathbb Z^2$, is a solution for $x^2+24y^2=24s+1$, then $x\neq 0$ and, since $24s+1=pm$ is not a perfect square, we also have  $y\neq 0$. Then, as in the proof of Corollary \ref{Cmain},   $M_i(pm)\equiv N_i(pm)\equiv 0\pmod 4$. 
 
 Since $pm$ is not a perfect square and $pm \equiv 1 \pmod{24}$, 
  to prove that  $b_3(2s)\equiv 0 \pmod 2$, we need to show that $$b_3(2s)\equiv a(s)\equiv \frac{1}{4}\Big(M_1(24s+1)-M_2(24s+1)\Big)\pmod 2.$$ Hence, we need to show that $$M_1(pm)-M_2(pm)\equiv 0\pmod 8.$$

 If $m\neq 1$, it follows from \eqref{N} that $N_1(pm)\equiv 0\pmod 8$. 
As previously mentioned, if $u$ is square free then $N_i(u)=M_i(u)$.
If $u=gh^2$ with $g$ square free, then $M_i(u)=\sum_{d|h}N_i(gd^2)$. \smallskip

 \noindent \underline{Case I:} If $m$ is not a perfect square, then $pm=pgh^2$ for some positive integers $h, g$ with $g$ square free. Then $$M_i(pm)=M_i(pgh^2)=\sum_{d|h}N_i(pgd^2).$$ Since $pm=pgh^2\equiv 1 \pmod{24}$, $\gcd(h,6)=1$ and $h^2\equiv 1 \pmod{24}$. Then $pg\equiv 1\pmod{24}$ and thus $pgd^2\equiv 1 \pmod{24}$ for all $d\mod h$.
 By Corollary \ref{Cmain}, $N_i(pgd^2)\equiv 0\pmod 8$ for all $d\mid h$ and thus $M_i(pm)\equiv 0 \pmod 8$.

 \smallskip

 \noindent \underline{Case II:} If $m$ is a perfect square, since $\gcd(m,6)=1$, we have $m\equiv 1 \pmod{24}$. Then $p \equiv 1 \pmod{24}$ and $N_1(p)=4$. We have  $$M_i(pm)=\sum_{\substack{d|h\\ d>1}}N_i(pd^2)+N_i(p).$$ By Corollary \ref{Cmain}, $N_2(pd^2)\equiv 0\pmod 8$ for all $d\mid m$, $d>1$,  and it follows that $M_i(pm)\equiv N_i(p) \pmod 8$.
 
Since $p\in \mathcal P$, $N_2(p)>0$. Since $N_2(p)\leqslant N_1(p)$ and $N_2(p)\equiv 0\pmod 4$, we have $N_1(p)=N_2(p)=4.$  Therefore $M_1(pm)-M_2(pm)\equiv 0 \pmod 8$. This concludes the proof. 

\subsubsection{The Dirichlet density of $\mathcal P$}

In this subsection we  determine the Dirichlet density of the set $\mathcal P$. Since the density is positive,  Theorem \ref{Tmain} holds for infinitely many primes. Moreover, we show that Theorem \ref{Tmain} holds for infinitely many primes congruent to $k \pmod{24}$ for each $k\in \{1,5,7,11\}$. 

\begin{definition} 
Let $\Sigma\subset \mathbb{N}$ be a subset of prime numbers. Then the Dirichlet density, $\delta(\Sigma)$, of $\Sigma$ is defined as 
$$\delta(\Sigma):=\lim_{s\to 1^{+}}\frac{\sum_{p\in\Sigma}p^{-s}}{\sum_{p}p^{-s}},$$
where the sum in the denominator is over all primes. 
\end{definition}
Using basic facts about the Riemann zeta function, it can be shown that $$\delta(\Sigma)=\lim_{s\to 1^{+}}\frac{\sum_{p\in\Sigma}p^{-s}}{-\log(s-1)}.$$

For  $k\in\{1,5,7,11\}$, let $\mathcal P_k=\{p\in \mathcal P: p\equiv k\pmod{24}\}.$

\begin{proposition} 
The Dirichlet density $\delta(\mathcal{P})=\frac{1}{6}.$ Furthermore $delta(\mathcal{P}_{k}))=\frac{1}{24}$ for $k\in\{1,5,7,11\}$.
\end{proposition}
\begin{proof} 

By Lemma \ref{L_j=1}, 
$$\mathcal{P}_1=\{p:\text{there exist $x,y\in\mathbb{Z}$ such that $x^2+216y^2=p$}\}.$$ 
By \cite[Theorem 9.12]{Cox} we have $$\delta(\mathcal{P}_1)=\frac{1}{2h(-864)}=\frac{1}{24}.$$

If there are $x,y\in\mathbb{Z}$ such that $\gcd(x,y)=1$ and $x^2+216y^2=4p$, then $x=2z$ for some $z\in \mathbb Z$,  and  $z^2+54y^2=p$. However if $(z,y)$ is a solution to $z^2+54y^2=p$, then indeed $(x,y)=(2z,y)$ is a solution to $x^2+216y^2=4p$ but we might have $\gcd(x,y)\neq 1$. It might be that $2|y$, in this case $z^2+216(y/2)^2=p$.  
Therefore,
\begin{align*}
\mathcal{P}_7=&\{p:\text{there exist $z,y\in\mathbb{Z}$ such that $z^2+54y^2=p$, $2\nmid y$}\}\\
=&\{p:\text{there exist $z,y\in\mathbb{Z}$ such that $z^2+54y^2=p$}\}\\ 
&- \{p:\text{there exist $x,y\in\mathbb{Z}$ such that $x^2+216y^2=p$}\}
\end{align*}
and by \cite[Theorem 9.12]{Cox} we have $$\delta(\mathcal{P}_7)=\frac{1}{2h(-216)}-\frac{1}{2h(-864)}=\frac{1}{24}.$$
Similarly, 
$$\mathcal{P}_5\cup \mathcal{P}_{11}=\{p:\text{there exist $z,y\in\mathbb{Z}$ such that $2z^2+27y^2=p$}\},$$
and by \cite[Theorem 9.12]{Cox} we have $$\delta(\mathcal{P}_5\cup\mathcal{P}_{11})=\frac{1}{2h(-216)}=\frac{1}{12}.$$
If $p\equiv 11 \pmod{24}$ and $z,y$ satisfy $2z^2+27y^2=p$, then $z$ must be even. Hence,  
$$\mathcal{P}_{11}=\{p:\text{there exist $z,y\in\mathbb{Z}$ such that $8z^2+27y^2=p$}\},$$
and by \cite[Theorem 9.12]{Cox} we have $$\delta(\mathcal{P}_{11})=\frac{1}{2h(-864)}=\frac{1}{24}.$$
This implies that $$\delta(\mathcal{P}_5)=\delta(\mathcal{P}_5\cup \mathcal{P}_{11})-\delta(\mathcal{P}_{11})=\frac{1}{12}-\frac{1}{24}=\frac{1}{24}.$$
Therefore, 
$$\delta(\mathcal{P})=\delta(\mathcal{P}_1)+\delta(\mathcal{P}_7)+\delta(\mathcal{P}_5\cup \mathcal{P}_{11})=\frac{1}{24}+\frac{1}{24}+\frac{1}{12}=\frac{1}{6}.$$
\end{proof}

Hence, for each $k\in \{1,5,7,11\}$, there are infinitely many primes congruent to $k$ modulo $24$  satisfying the congruences of Theorem \ref{Tmain}.

We conclude this section by showing that for $p\in \mathcal P$ we have $b_3\left( \frac{p^2-1}{12} \right) \equiv 1 \pmod 2$, which was noted in Remark \ref{R2}.

\begin{proposition} \label{Pp^2} A prime $p$ belongs to $\mathcal P$ if and only if the Diophantine equation $x^2+216y^2=p^2$ has primitive solutions.
\end{proposition}

\begin{proof}
 If $(x_1, y_1)$ is a primitive solution for $x^2+216y^2=jp$, where $j=1, 4, 8$, an easy calculation shows that \begin{align*}x'& := \frac{x_1^2-216y_1^2}{j}\\ y'& := \frac{2x_1y_1}{j}\end{align*} is a primitive solution  for $x^2+216y^2=p^2$. 
 
Conversely, let $x,y$ be coprime integers such that $x^2+216y^2=p^2$.  Then the pair $(X,Y):=(x,3y)$ is such that $X^2+24Y^2=p^2$ and $\gcd(X,Y)=1$. 

Suppose $p\equiv 1\pmod{24}$. Since $p\equiv 1\pmod{24}$, we have that $N_1(p)=4$. So, there exist  $a,b\in \mathbb Z$ such that  $a^2+24b^2=p$ and $\gcd(a,b)=1$. This implies that the pair $(u,v)$ defined by 
$$(u+2\sqrt{-6}v):=(a+2\sqrt{-6}b)^2=a^2-24b^2+2\sqrt{-6}(2ab)$$ is such that  $u^2+24v^2=p^2$. Since $\gcd(a,b)=1$ and $\gcd(a,6)=1$, it follows that $\gcd(u,v)=1. $
Then,  $(x,3y)=(X,Y)\in\{(\pm u,\pm v)\}.$
Since $v=2ab$, we must have $ab\equiv 0\pmod{3}$. Since $3\nmid a$, we have $3|b$.  Then, we can write $b=3c$ for some $c\in \mathbb Z$. This implies $a^2+216c^2=p$ and $(a, c)$ is a primitive solution for $x^2+216y^2=jp$ with $j=1$.
%We have shown that for $p\equiv 1\pmod{24}$ if there exist $(x,y)$ such that 
%$x^2+216y^2=p^2$, then there exist $(a,c)$ such that $a^2+216c^2=p$.

Next, we consider  the case $p\equiv 5,7,11\pmod{24}$. 
 For $\gcd(w,6)=1$,  the number of primitive representations of $w$ by a positive, reduced, primitive quadratic form of discriminant $-24$ is 
 $$N(-24,w)=2\prod_{p|w}\left(1+\left(\frac{-6}{p}\right)\right).$$
 Since $h(-24)=2$, we have two quadratic forms of discriminant $-24$: $x^2+6y^2$ and $2x^2+3y^2$.  Therefore,
\begin{align*}
      N(-24,w)=&|\{(x,y)\in\mathbb{Z}:x^2+6y^2=w,\gcd(x,y)=1\}|\\
       &+|\{(x,y)\in\mathbb{Z}:2x^2+3y^2=w,\gcd(x,y)=1\}|.
  \end{align*}
We  define
\begin{align*}N_3(w)& :=|\{(x,y)\in\mathbb{Z}:x^2+6y^2=w,\gcd(x,y)=1\}|\\
N_4(w)& :=|\{(x,y)\in\mathbb{Z}:2x^2+3y^2=w,\gcd(x,y)=1\}|.\end{align*}

 We note that $6$ is an idoneal number. If $w\equiv 1,7\pmod{24}$ then only $x^2+6y^2=w$ has solutions and so $N(-24,w)=N_3(w)$, and if $w\equiv 5,11\pmod{24}$ then only $2x^2+3y^2=w$ has solutions and so $N(-24,w)=N_4(w)$. 

  First let $p\equiv 7\pmod{24}$. % and let $x,y$ be coprime integers such that $x^2+216y^2=p^2$. Then $(X,Y):=(x,3y)$ is a primitive solution for  $x^2+24y^2=p^2$. Since $p\equiv 7\pmod{24}$ 
  We have $N_3(p)=4$. Hence, there exist  $a,b\in \mathbb Z$ such that  $a^2+6b^2=p$ and $\gcd(a,b)=1$.   This implies that the pair $(u,v)$ defined by 
 $$u+2\sqrt{-6}v:=(a+\sqrt{-6}b)^2=a^2-6b^2+2\sqrt{-6}ab$$ is such that $u^2+24v^2=p^2$. Moreover, $\gcd(u,v)=1$.
 Then, $(x,3y)=(X,Y)\in \{(\pm u,\pm v)\}.$
 Since $v=ab$ we must have $ab\equiv 0\pmod{3}$. Since $3\nmid a$ we have $3|b$. So we can write $b=3c$ for some $c\in \mathbb Z$. This implies $a^2+54c^2=p$ or, equivalently, $(2a)^2+216c^2=4p$. Next note that $2\nmid c$ because of $p\equiv 7\pmod{24}$, which implies $\gcd(2a,c)=1$. Hence, $(2a, c)$ is a primitive solution for $x^2+216y^2=jp$ with $j=4$.
  
  Finally, let $p\equiv 5,11\pmod{24}$.  %and assume that $(x,y)$ is a primitive solution for $x^2+216y^2=p^2$. Then $(X,Y):=(x,3y)$ is a primitive solution for $x^2+24y^2=p^2$. Since $p\equiv 5,11\pmod{24}$ 
  We have $N_4(p)=4$. Hence, there exist  $a,b\in \mathbb Z$ such that  $2a^2+3b^2=p$ and $\gcd(a,b)=1$. This implies that the pair $(u,v)$ defined by 
 $$u+2\sqrt{-6}v:=(\sqrt{2}a+\sqrt{-3}b)^2=2a^2-3b^2+2\sqrt{-6}ab$$ is such that $u^2+24v^2=p^2$. Moreover, $\gcd(u,v)=1$.
 Then, $(x,3y)=(X,Y)\in \{(\pm u,\pm v)\}.$
 Since $v=ab$ we must have $ab\equiv 0\pmod{3}$. Since $3\nmid a$ we have $3|b$. So we can write $b=3c$ for some $c\in \mathbb Z$. This implies $2a^2+27c^2=p$ or, equivalently, $(4a)^2+216c^2=8p$. Hence, $(4a, c)$ is a primitive solution for $x^2+216y^2=jp$ with $j=8$.
  
\end{proof}

\begin{corollary} If $p\in \mathcal P$, then $b_3\left( \frac{p^2-1}{12} \right) \equiv 1 \pmod 2$.
\end{corollary}

\begin{proof} To show that for $p\in \mathcal P$ we have $b_3\left( \frac{p^2-1}{12} \right) \equiv 1 \pmod 2$,  we need to show that $x^2+24y^2=p^2$ has an odd number of solutions $(x,y)\in \mathbb N^2$ such that $y=0$ or $3\nmid y$. 
If $p\in \mathcal P$, then 
 $-6$ is a quadratic residue modulo $p$ and from \eqref{M1} we have $M_1(p^2)=6$. Of the six solution for $x^2+24y^2=p^2$,  there are only two solutions in $\mathbb N^2$, one of which is $(p,0)$. By Proposition \ref{Pp^2}, the other solution $(x,y)$ satisfies $y\neq 0$ and $3\mid y$. It follows that  $a\left( \frac{p^2-1}{24} \right)=1$ and thus
 $b_3\left( \frac{p^2-1}{12} \right) \equiv 1 \pmod 2$. 
  \end{proof}

\section{Reversals of Euler type theorems}\label{SL}

If $S$ is a subset of positive integers, we denote by $p(S;n)$ the number of partitions of $n$ with parts from $S$. If $k\geqslant 2$, we denote  by $q_k(S;n)$ the number of partitions of $n$ with parts from $S$, each part occurring at most $k-1$ times. If $k=2$, we write $q(S;n)$ for $q_2(S;n)$. Thus, $q(S;n)$ is the number of distinct partitions of $n$ with parts from $S$. 

A pair  $(S_1, S_2)$ of subsets of positive integers is called an Euler pair if $q(S_1; n)=p(S_2;n)$ for all $n\geqslant 0$. 
 Andrews \cite{A69}  proved that $(S_1, S_2)$ is an Euler pair if and only if $2S_1\subseteq S_1$ and $S_2=S_1\setminus 2S_1$.
 
 This notion has been generalized by Subbarao \cite{S71}. Given $k \geqslant 2$, the pair $(S_1, S_2)$ is called an Euler pair of order $k$ if $q_k(S_1;n)=p(S_2;n)$ for all $n\geqslant 0$. 
  Subbarao proved that  $(S_1, S_2)$ is  an Euler pair of order $k$ if and only if $kS_1\subseteq S_1$ and $S_2=S_1\setminus kS_1$.

\begin{example}[Subbarao \cite{S71}]\label{eg:epair}
Let 
\begin{align*}
S_1 &= \{ m \in \N : \ m \equiv 1  \pmod  2 \};\\
S_2 &= \{ m \in \N : \ m \equiv \pm 1  \pmod  6 \}.
\end{align*}

\noindent Then $(S_1, S_2)$ is an Euler pair of order 3.
\end{example}

We note that Glaisher's bijection  used to prove Euler's identity $q(\N, n)=p(2\N-1, n)$ can be generalized to prove the corresponding partition identity for any Euler pair of order $r$. 

Next, we consider a reversal of the partition identity given by the Euler pair $(S_1, S_2)$. We denote by $p_e(S;n)$, respectively $p_o(S;n)$, the number of partitions of $n$ with parts in $S$ and an even, respectively odd, number of parts. We use analogous notation for other partition numbers.

\begin{theorem} \label{rev} Let $(S_1, S_2)$ be an Euler pair. For $n\geqslant 0$, we have $$q_e(S_2;n)-q_o(S_2;n)=p_e(S_1;n)-p_o(S_1,n).$$
\end{theorem}

\begin{proof}[Analytic proof]

We have
\begin{align*} \sum_{n\geqslant 0}(p_e(S_1;n)-p_o(S_1,n))q^n& =\prod_{j\in S_1}\frac{1}{1+q^j} =\prod_{j\in S_1}\frac{1-q^j}{1-q^{2j}}\\ & =\prod_{j\in S_1}(1-q^j) \prod_{j\in 2S_1}\frac{1}{1-q^{j}}\\ & = \prod_{j\in S_1}(1-q^j) \prod_{j\in S_1\setminus S_2}\frac{1}{1-q^{j}}\\ & = \prod_{j\in S_2}(1-q^j) \end{align*}
This completes the proof. 
\end{proof}

\begin{proof}[Combinatorial proof] Our combinatorial proof of this result is similar to that in \cite{G75}. Since $(S_1, S_2)$ is an Euler pair, it follows that $2S_1\subseteq S_1$ and $S_2=S_1\setminus 2S_1$. Given a partition $\lambda$ with parts in $S_1$, denote by $\ell(\lambda)$ the number of parts in $\lambda$ and by $\ell_2(\lambda)$ the number of even parts  from $2S_1$. 

Let $\mathcal P'(S_1, n)$ be the set of partitions $\lambda$ of $n$ with parts in $S_1$ such that $\lambda$ has  at least one repeated part or at least one  part from $2S_1$ and denote by $\mathcal P'_e(S_1, n)$, respectively $\mathcal P'_o(S_1, n)$, the subset of partitions in $\mathcal P'(S_1, n)$ with $\ell_2(\lambda)$ even, respectively odd. We define an involution $\varphi$ on $\mathcal P'(S_1, n)$ that reverses the parity of $\ell_2(\lambda)$.

Start with $\lambda\in \mathcal P'(S_1, n)$. Following Gupta's notation \cite{G75}, we denote by $r$ the largest repeated part of $\lambda$ and by $e$ the largest part of $\lambda$ that is from $2S_1$. If $r$ or $e$ do not exist, we set them equal to $0$. 
\begin{enumerate}
\item If $2r>e$,  we define $\varphi(\lambda)$ to be the partition obtained from $\lambda$ by replacing two parts equal to $r$ by a single part equal to $2r$. 

\item If $2r\leqslant e$, we define $\varphi(\lambda)$ to be the partition obtained from $\lambda$ by replacing one part equal to $e$ by  two  parts equal to $e/2$. 
\end{enumerate}

Since $\varphi: \mathcal P'(S_1, n)\to \mathcal P'(S_1, n)$ is an involution that reverses the parity of $\ell_2(\lambda)$,  we have  that $|\mathcal P'_e(S_1, n)|=|\mathcal P'_o(S_1, n)|$. This proves that $$q_e(S_2;n)-q_o(S_2;n)=  p_e(S_1, n)-p_o(S_1, n).$$
\end{proof}

\begin{corollary}\label{cor-rev} Let $(S_1, S_2)$ be an Euler pair such that $S_2\subseteq 2\N-1$. For $n\geqslant 0$, we have $$(-1)^nq(S_2;n)=p_e(S_1;n)-p_o(S_1,n).$$

\end{corollary}

\begin{proof} If $S_2$ consists of odd integers and $\lambda$ has parts in $S_2$, then $\ell(\lambda)\equiv n \pmod 2$. Then, $q_e(S_1;n)-q_o(S_1;n)=(-1)^n q(S_1;n)$. \end{proof}

\begin{remark} It is possible to have an Euler pair $(S_1, S_2)$  such that $S_2\not\subseteq 2\N-1$. For example, $S_1=\{m\in \N \mod m\equiv 2,4,5 \pmod 6\}$ and $S_2=\{m\in \N \mid m\equiv 2,5,11 \pmod{12}\}$ is an Euler pair and $q(S_1; n)=p(S_2;n)$ is G\"ollnitz's identity. 
\end{remark}

\begin{remark} \label{RSchur} If $S_1=\{n\in \mathcal N \mid n \equiv \pm 1\pmod 3\}$ and $S_2=\{n\in \mathcal N \mid n \equiv \pm 1\pmod 6\}$, then $(S_1, S_2)$ is and Euler pair and $q(S_1;n)=p(S_2;n)$ is Schur's identity.  Thus, our combinatorial proof of Theorem \ref{rev} gives a combinatorial proof of \cite[Theorem 1.1]{M22}.
Moreover, in this case $p(S_1;n)=b_3(n)$ and $S_2 \subseteq 2\N-1$. Thus, by Corollary \ref{cor-rev}, $b_3(n)\equiv q(S_2;n) \pmod 2$. \end{remark}

%\begin{theorem} Let $(S_1,S_2)$ be an Euler pair of order $4$. For $n\geqslant 0$ we have $$q_e(S_2;n)-q_e(S_2;n)=p_e(S_1;n)-p_o(S_1,n).$$
%\end{theorem} Suppose $(S_1,S_2)$ is an Euler pair of order 4. Then $q_4(S_1;n)=p(S_2;n)$. Here $q_4(S_1;n)$ is the number of partitions of $n$ with parts in $S_1$ and no part is repeated more than $3$ times. From Subbarao, $4S_1\subseteq S_1$ and $S_2=S_1\setminus 4S_1$. 

%We consider $p_{e-o}(S_1;n)$. The generating function is \begin{align*}\sum_{n\geqslant 0}p_{e-o}(S_1;n)q^n& = \prod_{j\in S_1}\frac{1}{1+q^j}\\& =  \prod_{j\in S_1}\frac{1}{1+q^j} \frac{(1-q^j)(1+q^{2j})}{(1-q^j)(1+q^{2j})}\\ & = \prod_{j\in S_1} (1-q^j)(1+q^{2j}) \prod_{j\in S_1} \frac{1}{1-q^{4j}} \\ & = \prod_{j\in S_1} (1-q^j)(1+q^{2j}) \prod_{j\in 4S_1} \frac{1}{1-q^{j}} \\ & = \prod_{j\in S_1} (1-q^j)\prod_{j\in S_1\setminus S_2} \frac{1}{1-q^{j}}   \prod_{j\in S_1}(1+q^{2j}) \\ & = \prod_{j\in S_2} (1-q^j)\  \prod_{j\in S_1}(1+q^{2j})  \end{align*}

\begin{theorem}
	Let $r$ be an odd integer. For $n\geqslant0$, 
	$$
	q_r(\mathbb{N};n) \equiv q(M-r\,M;n) \pmod 2,
	$$
	where $M=2\mathbb{N}-1$,
\end{theorem}

\begin{proof}
	From  Subbarao's theorem (also Glaisher's identity), we can write
	\begin{align*}
	\sum_{n=0}^\infty q_r(\mathbb{N};n)\,q^n = \sum_{n=0}^\infty p(\mathbb{N}-r\,\mathbb{N};n)\, q^n = \frac{(q^r;q^r)_\infty}{(q;q)_\infty}.
	\end{align*}
	On the other hand, elementary techniques in the theory of partitions give the following generating
	function
	\begin{align*}
	F(z;q) &:= \prod_{k=0}^\infty \frac{1}{(1-zq^{kr+1})(1-zq^{kr+2})\cdots (1-zq^{kr+r-1})} \\
	&= \sum_{m=0}^\infty \sum_{n=0}^\infty p(\mathbb{N}-r\,\mathbb{N};n,m)\,z^m\,q^n,
	\end{align*}
	where $p(\mathbb{N}-r\,\mathbb{N};n,m)$ is the number of partitions of $n$ with $m$ parts all taken from
	$\mathbb{N}-r\,\mathbb{N}$. 
	We can write
	\begin{align*}
	\sum_{n=0}^\infty (p_{e}(\mathbb{N}-r\,\mathbb{N};n)- p_{o}(\mathbb{N}& -r\,\mathbb{N};n))\,q^n   = F(-1;q).\end{align*}
 On the other hand, 
 \begin{align*}
 F(-1;q) & = \frac{(-q^r;q^r)_\infty}{(-q;q)_\infty} = \frac{(q;q^2)_\infty}{(q^r;q^{2r})_\infty} \\
	& = \frac{ (q;q^{2r})_\infty(q^3;q^{2r})_\infty\cdots (q^r;q^{2r})_\infty \cdots (q^{2r-1};q^{2r})_\infty}{(q^r;q^{2r})_\infty} \\
	& = \sum_{n=0}^\infty (-1)^n\,q(M-rM;n)\,q^n.
	\end{align*}
Thus we deduce that $p_{e}(\mathbb{N}-r\,\mathbb{N};n)- p_{o}(\mathbb{N}-r\,\mathbb{N};n)= (-1)^n\,q(M-rM;n)$ and the proof follows easily.
\end{proof}

An overpartition of $n$ is a partition of $n$ in which the first occurrence of a part may be overlined. 
 If $(S_1, S_2)$ is an Euler pair of order $r$, 
we denote by $\overline p_r(S_1, n)$ the number of overpartitions of $n$ with parts in $S_1$ and only parts  from $S_1\setminus S_2$ may be overlined. 

\begin{theorem}If $(S_1, S_2)$ is an Euler pair of order $r$, then $$
	q_r(S_1;n) \equiv  \overline p_r(S_1, n) \pmod 2,
	$$
\end{theorem}

\begin{proof} Since $(S_1, S_2)$ is an Euler pair of order $r$, we have $q_r(S_1;n)=p(S_2;n)$. Moreover,  $rS_1\subseteq S_1$ and $S_2=S_1\setminus rS_1$. Thus $$q_r(S_1;n)=p(S_2;n)=p_e(S_2;n)+p_o(S_2;n)\equiv p_e(S_2;n)-p_o(S_2;n)\pmod 2.$$
We have 
\begin{align*}\sum_{n=0}^\infty (p_e(S_2;n)-p_o(S_2;n))q^n& = \prod_{i\in S_2}\frac{1}{1+q^i} = \prod_{i\in S_1\setminus rS_1}\frac{1}{1+q^i} = \prod_{i\in S_1}\frac{1+q^{ri}}{1+q^i}\\ & \equiv\prod_{i\in rS_1}(1+q^{i}) \prod_{i\in S_1}\frac{1}{1-q^i}  \pmod 2
\end{align*}
The first product generates the overlined parts and the second product generates the non-overlined parts. Since $S_2=S_1\setminus rS_1$, this concludes the proof. 
\end{proof}

\section{Concluding remarks}\label{S4}

In this article we found infinitely many new arithmetic progressions $\{s_{p,n}=p^2n+p\alpha-24_p^{-1}\}_{n\geqslant 0}$, where $p\in \mathcal P$ and $\alpha\neq \lfloor p/24\rfloor$ is a residue modulo $p$. For small $p\in \mathcal P$, the congruneces  could be proved using modular forms. However,  the general case required an intricate analysis of the Diophantine equation \begin{equation}\label{eqn_D} x^2+216y^2=pm,\end{equation} where $p\in \mathcal P$, $pm\equiv 1\pmod{24}$ and $p\nmid m$. The difficulty in this analysis  is caused by the fact that $216$ is not an idoneal or Euler convenient number. To prove Theorem \ref{Tmain}, we proved that the number of primitive solutions for \eqref{eqn_D} is divisible by $8$. 

Here, for $m\equiv 1\pmod{24}$, we conjecture an exact formula for $N_2(m)$, the number of primitive solutions  for \begin{equation}\label{eqn_D1} x^2+216y^2=m.\end{equation}

From \cite[Lemma 1.7]{Cox}, if there is a prime $p\mid m$ with  $p\equiv 13,17,19,23 \pmod{24}$, then $N_2(m)=0$.

Denote by $\mathcal Q$ be the set of prime numbers not in $\mathcal{P}$. 
Given a positive  integer $m$,   write  
$$m=p_1^{a_1}p_2^{a_2}\cdots p_k^{a_k}q_1^{b_1}q_2^{b_2}\cdots q_n^{b_n},$$
where $p_1,p_2,\ldots,p_k \in \mathcal P$  and  $q_1,q_2,\ldots,q_n \in \mathcal Q$.
Let  $$B(m):=|\{b_i\in \mathcal Q: b_i\equiv 0 \pmod 3\}|.$$

 \begin{conjecture}\label{solconjecture} Let $m$ be a positive integer such that $m\equiv 1\pmod{24}$  and all prime divisors of $m$ are congruent to $1, 5,7, 11$ modulo $24$. With the above notation we have 
\begin{equation*}
\begin{split}
N_2(m)=&2^{B(m)+k}\cdot \frac{2^{n-B(m)+1}+4\cdot(-1)^{n-B(m)}}{3}\\
=&\frac{2^{n+k+1}+2^{B(m)+k+2}\cdot(-1)^{n-B(m)}}{3}.
\end{split}
\end{equation*}
\end{conjecture}

We note that Conjecture \ref{solconjecture} implies Theorem \ref{Th1}. \smallskip

As noted in Remark \ref{RSchur}, it follows from  Corollary \ref{cor-rev} that $$b_3(n)\equiv q(S_2;n)\pmod 2,$$ where $q(S_2;n)$ is the number of distinct $3$-regular partitions with all parts odd.
Then the congruences of Theorems \ref{Th1}, \ref{Th2}, and \ref{Tmain} hold for $q(S_2;n)$.

The results of section \ref{SL} led us to search for parity results for $b_{3,e}(n)$ (respectively $b_{3,e}(n)$), the number of $3$-regular partitions with an even  (respectively odd) number of parts.

We end the article with the following congruence conjectures for which there is substantial numerical evidence.

The first conjecture is similar to Theorem \ref{Th1}. 
\begin{conjecture} \label{conj 4}
	For $n\geqslant 0$ and $p\equiv 13,17,19,23 \pmod {24}$,
	$$b_{3,e}\big(2\,(p^2\, n+p\alpha-24^{-1})\big) \equiv 0 \pmod 2,$$
	$$b_{3,o}\big(2\,(p^2\, n+p\alpha-24^{-1})\big) \equiv 0 \pmod 2,$$
	where $1\leqslant \alpha \leqslant p-1$ and $24^{-1}$ is taken modulo $p^2$. 
\end{conjecture}

The next conjecture is similar to Theorem \ref{Tmain}.
\begin{conjecture} \label{C1.4}
	For $n\geqslant 0$, and $p\in\mathcal{P}$ with $p\equiv 1,7 \pmod{24}$,
	$$b_{3,e}\big(2\,(p^2\,n+p\,\alpha-24_p^{-1})\big) \equiv 0 \pmod 2,$$
	$$b_{3,o}\big(2\,(p^2\,n+p\,\alpha-24_p^{-1})\big) \equiv 0 \pmod 2,$$
	where $0\leqslant \alpha < p$, $\alpha \neq \lfloor p/24 \rfloor$.  
\end{conjecture}

%{\color{red} Mai scriem ceva despre verificari?}
%\ccom{The conjectures have been verified for all values $2\,(p^2\, n+p\alpha+\beta)\leqslant 2\cdot 10^7$.}

To prove Conjectures \ref{conj 4} and \ref{C1.4}, one could try to make use of 
 Corollary \ref{cor-rev} which gives  $$2b_{3,e}(m)=b_3(m)+q(S_2;m).$$ 
Thus, to prove the conjectures, it is enough to show that $$b_3(m)\equiv q(S_2;m) \pmod 4$$ for the suitable values of $m$.

A different approach would be to use \cite[Theorem 1.4]{M22}, which states that \begin{align*}b_{3,e}(m)& =p(S_3;m),\end{align*} where
\begin{align*}S_3& = \{ n \in \mathbb N\mid n\not \equiv 0, \pm 1, \pm 10, \pm 11, 12 \mod{24}\}, \end{align*} and try to prove the conjectures for the function $p(S_3;m)$.  

It is likely that similar conjectures can be formulated as companions to Theorem \ref{Th1}.
For example, there is considerable numerical evidence for the next conjecture which fits into the case $k=0$ of Theorem \ref{Th1} (2). 
\begin{conjecture}\label{conj 6}
	For $n\geqslant 0$, and $p\in\{7,31\}$,
	$$b_{3,e}\big(2\,(p^3\, n+p^2\alpha-24^{-1})\big) \equiv 0 \pmod 2;$$
	$$b_{3,o}\big(2\,(p^3\, n+p^2\alpha-24^{-1})\big) \equiv 0 \pmod 2;$$
	where $0\leqslant \alpha < p$, $\alpha \not\equiv -24^{-1} \pmod p$, and $24^{-1}$ is taken modulo $p^2$.
\end{conjecture}

\end{document}